\documentclass[11pt]{article}

\usepackage{amsmath,amssymb,amsthm,mathtools}
\usepackage[margin=1in]{geometry}

\newtheorem{theorem}{Theorem}[section]
\newtheorem{lemma}[theorem]{Lemma}
\newtheorem{corollary}[theorem]{Corollary}

\newcommand{\N}{\mathbb N}
\newcommand{\card}[1]{\lvert #1\rvert}

\title{Dense Sets with Prescribed Difference Representation Functions}
\author{}
\date{}

\begin{document}

	\author{Wu-Xia Ma$^a$\footnote{ Corresponding author.} and Bing-Ling Wu$^{b,}$\\
		\small School of Science, Nanjing University of Posts and Telecommunications,\\
		\small Nanjing 210023, PR China\\
		\small E-mails: mawuxia@126.com$^a$, wubingling911219@163.com$^b$}
	\maketitle
	
	\begin{abstract}
		For a set $A$ of positive integers, let $d_A(n)$ denote the number of
		representations of $n$ as a difference of two elements of $A$.  We
		prove that, if $\liminf_{n\to\infty}f(n)\geq2$, then for every Sidon
		set $B$ and every positive function $\omega(x)$ tending to infinity
		there is a set $A\subseteq\N$ such that $d_A(n)=f(n)$ for all
		$n\in\N$ and $|A(x)-B(x/2)|\leq\omega(x)$ for all $x\geq1$.  We also
		give two extensions that allow $f(n)=1$ for infinitely many $n$.  For
		an infinite Sidon set $B$, the same density scale $B(x/2)$ is retained
		when $f$ is at least $2$ on arbitrarily long intervals centered at
		elements of $2B$.  If $\{n:f(n)\geq2\}$ contains arbitrarily long
		intervals, the density scale $B(x/3)$ is obtained for every Sidon set
		$B$.  The proofs use remote insertion of new representations together
		with controlled deletions from a Sidon core.  As a quantitative
		consequence, Ruzsa's dense infinite Sidon sequence gives
		$A(x)=x^{\sqrt2-1+o(1)}$ whenever $\{n:f(n)\geq2\}$ contains
		arbitrarily long intervals.
	\end{abstract}
	
	\medskip
	\noindent\textbf{Keywords:} difference representation function; Sidon set;
	perfect difference set; dense set.
	
	\smallskip
	\noindent\textbf{2020 Mathematics Subject Classification:}
	11B13, 11B34.
	
	\section{Introduction and statements of the results}
	
	Throughout the paper, $\N=\{1,2,3,\ldots\}$.  If $A$ is a set of
	integers and $n$ is an integer, put
	\[
	d_A(n)=\card{\{(a,a')\in A^2:a-a'=n\}}.
	\]
	Thus $d_A(n)$ is the number of ordered representations of $n$ as a
	difference of two elements of $A$.  For $A\subseteq\N$ and $x\geq1$,
	write $A(x)=\card{A\cap[1,x]}$.  A set $A\subseteq\N$ is called a
	\emph{Sidon set} if $d_A(n)\leq1$ for every $n\in\N$, and a
	\emph{perfect difference set} if $d_A(n)=1$ for every $n\in\N$.
	
	The first quantitative constructions of perfect difference sets give
	sets $A$ with $A(x)\gg x^{1/3}$; see Pollington and Vanden
	Eynden~\cite{PollingtonVandenEynden}.  Pollington~\cite{Pollington}
	later obtained the bound $A(x)\gg(x\log x)^{1/3}$.  These estimates led
	Lev to ask at CANT 2004 whether a perfect difference set could satisfy
	\[
	A(x)\gg x^\delta\qquad\text{for some }\delta>\frac13;
	\]
	see also~\cite{Lev}.
	
	Cilleruelo and Nathanson~\cite{CillerueloNathanson2008} developed a
	construction of perfect difference sets from Sidon sets and proved the
	following result.
	
	\medskip
	\noindent
	\textbf{Theorem A}
	{\rm(\cite[Theorem 1.1]{CillerueloNathanson2008}).}
	For every Sidon set $B\subseteq\N$ and every function
	$\omega(x)\to\infty$, there exists a perfect difference set
	$A\subseteq\N$ such that
	\[
	A(x)\geq B(x/3)-\omega(x)
	\]
	for all sufficiently large $x$.
	\medskip
	
	Together with Ruzsa's construction of an infinite Sidon
	sequence~\cite{Ruzsa1998}, Theorem~A yields a perfect difference set
	$A$ for which
	\[
	A(x)\gg x^{\sqrt2-1-o(1)}.
	\]
	Cilleruelo and Nathanson also constructed a perfect difference set
	satisfying
	\[
	\limsup_{x\to\infty}\frac{A(x)}{\sqrt{x}}\geq\frac1{\sqrt2}.
	\]
	
	Representation functions and their inverse problems have been studied
	in several related settings; see
	\cite{AjtaiKomlosSzemeredi,CillerueloNathanson2013,Kruckeberg,
		Nathanson2003,Nathanson2005,Ruzsa1993,Stohr}.  Fang~\cite{Fang2020}
	considered difference and sum representation functions simultaneously.
	More precisely, if $f_1:\N\to\N$ and $f_2:\mathbb Z\to\N$ satisfy
	$\liminf_{n\to\infty}f_1(n)\geq2$ and
	$\liminf_{|n|\to\infty}f_2(n)\geq2$, he
	constructed a set $A\subseteq\mathbb Z$ for which $d_A=f_1$ on
	$\N$, the unordered two-term sum representation function of $A$ is
	$f_2$ on $\mathbb Z$, and
	$\limsup_{x\to\infty}A(-x,x)/\sqrt{x}\geq1/\sqrt2$, where
	$A(-x,x)=\card{A\cap[-x,x]}$.  In the setting
	of positive integers and density measured relative to a prescribed
	Sidon set, Fang~\cite{Fang2021} obtained the following theorem.
	
	\medskip
	\noindent
	\textbf{Theorem B} {\rm(\cite[Theorem 1]{Fang2021}).}
	Let $f:\N\to\N$ be increasing and satisfy $f(n)\geq2$ for every
	$n\in\N$.  For every Sidon set $B\subseteq\N$ and every function
	$\omega(x)\to\infty$, there exists a set $A\subseteq\N$ such that
	\[
	d_A(n)=f(n)\qquad(n\in\N)
	\]
	and
	\[
	B(x/3)-\omega(x)\leq A(x)\leq B(x/3)+\omega(x)
	\]
	for all sufficiently large $x$.
	\medskip
	
	More recently, Chen and Fang~\cite{ChenFang2027} proved the following
	density theorem for perfect difference sets.
	
	\medskip
	\noindent
	\textbf{Theorem C} {\rm(\cite[Theorem 1.1]{ChenFang2027}).}
	For every Sidon set $B\subseteq\N$ and every positive function
	$\omega(x)\to\infty$, there exists a perfect difference set
	$A\subseteq\N$ such that
	\[
	B(x/2)-\omega(x)\leq A(x)\leq B(x/2)+\omega(x)
	\qquad(x\geq1).
	\]
	\medskip
	
	The following is the basic form of our result.
	
	\begin{theorem}\label{thm:main}
		Let $f:\N\to\N$ satisfy
		$\liminf_{n\to\infty}f(n)\geq2$.  For every Sidon set
		$B\subseteq\N$ and every positive function $\omega(x)\to\infty$,
		there exists a set $A\subseteq\N$ such that
		\[
		d_A(n)=f(n)\qquad(n\in\N)
		\]
		and
		\[
		B(x/2)-\omega(x)\leq A(x)\leq B(x/2)+\omega(x)
		\qquad(x\geq1).
		\]
	\end{theorem}
	
	No monotonicity is assumed for either $f$ or $\omega$.  The hypothesis
	is equivalent to $f(n)\geq2$ for all sufficiently large $n$, so
	finitely many prescribed values may equal $1$.
	
	An arbitrary positive-valued function need not be a difference
	representation function.  For example, define $g:\N\to\N$ by
	$g(1)=3$ and $g(n)=1$ for $n\geq2$.  If $d_A=g$, then the equality
	$d_A(1)=3$ gives three distinct integers $x_1<x_2<x_3$ such that
	$x_i,x_i+1\in A$ for $i=1,2,3$.  With $v=x_3-x_1\geq2$, we obtain
	\[
	v=x_3-x_1=(x_3+1)-(x_1+1),
	\]
	and hence $d_A(v)\geq2$, contrary to $g(v)=1$.  More generally, if
	$a,a+u,b,b+u\in A$ with $a<b$, then $b-a$ has the two representations
	$(b,a)$ and $(b+u,a+u)$.  Thus repeated representations of one
	difference impose restrictions on other representation numbers.
	
	The condition in Theorem~\ref{thm:main} can nevertheless be relaxed.
	For $f:\N\to\N$, put $S_f=\{n\in\N:f(n)\geq2\}$.  Long intervals of
	this type also occur in the work of Chen and Fang~\cite{ChenFang2011}
	on simultaneous difference and sum representation functions.  The
	first extension retains the density scale $B(x/2)$ by requiring such
	intervals near the elements of the core $2B$.
	
	\begin{theorem}\label{thm:half}
		Let $B\subseteq\N$ be an infinite Sidon set and let $f:\N\to\N$.
		Suppose that, for every $L\in\N$, there are infinitely many $b\in B$
		such that
		\[
		[2b-L,2b+L]\cap\N\subseteq S_f.
		\tag{1.1}\label{eq:compatibility}
		\]
		Let $\omega:[1,\infty)\to(0,\infty)$ satisfy
		$\omega(x)\to\infty$.  Then there exists $A\subseteq\N$ such that
		$d_A(n)=f(n)$ for every $n\in\N$ and
		\[
		B(x/2)-\omega(x)\leq A(x)\leq B(x/2)+\omega(x)
		\qquad(x\geq1).
		\]
	\end{theorem}
	
	If $\liminf_{n\to\infty}f(n)\geq2$, then, for every $L\in\N$, all
	sufficiently large $b\in B$ satisfy \eqref{eq:compatibility}.
	Consequently, Theorem~\ref{thm:main} follows from
	Theorem~\ref{thm:half} when $B$ is infinite; the finite case follows
	from Lemma~\ref{lem:finite-core} in Section~5.
	
	The next extension imposes a condition on $f$ alone and uses the Sidon
	core $3B$.
	
	\begin{theorem}\label{thm:third}
		Let $f:\N\to\N$ and suppose that $S_f$ contains arbitrarily long
		intervals of consecutive integers.  Let $B\subseteq\N$ be a Sidon set
		and let $\omega:[1,\infty)\to(0,\infty)$ satisfy
		$\omega(x)\to\infty$.  Then there exists $A\subseteq\N$ such that
		$d_A(n)=f(n)$ for every $n\in\N$ and
		\[
		B(x/3)-\omega(x)\leq A(x)\leq B(x/3)+\omega(x)
		\qquad(x\geq1).
		\]
	\end{theorem}
	
	Theorems~\ref{thm:half} and~\ref{thm:third} allow $f(n)=1$ for
	infinitely many $n$.  For example, Theorem~\ref{thm:third} applies
	when $f(n)=2$ on $[k!,k!+k]$ for every $k$ and $f(n)=1$ elsewhere.
	
	Ruzsa~\cite{Ruzsa1998} constructed an infinite Sidon set $B$ satisfying
	$B(x)=x^{\sqrt2-1+o(1)}$.  Theorem~\ref{thm:third} therefore has the
	following explicit density consequence.
	
	\begin{corollary}\label{cor:ruzsa}
		Let $f:\N\to\N$ and suppose that $S_f$ contains arbitrarily long
		intervals of consecutive integers.  Then there exists $A\subseteq\N$
		such that $d_A(n)=f(n)$ for every $n\in\N$ and
		$A(x)=x^{\sqrt2-1+o(1)}$ as $x\to\infty$.  In particular, this
		conclusion holds if $\liminf_{n\to\infty}f(n)\geq2$.
	\end{corollary}
	
	\begin{proof}
		Apply Theorem~\ref{thm:third} to Ruzsa's set with
		$\omega(x)=\log(x+1)$.  Then
		$A(x)=B(x/3)+O(\log(x+1))=x^{\sqrt2-1+o(1)}$.
	\end{proof}
	
	For reference, the hypotheses and density conclusions of the closest
	results may be summarized as follows.
	
	\begin{table}[ht]
		\centering
		\small
		\renewcommand{\arraystretch}{1.15}
		\begin{tabular}{@{}p{0.18\textwidth}p{0.48\textwidth}p{0.25\textwidth}@{}}
			\hline
			Result & Prescribed representation problem & Density conclusion \\
			\hline
			Fang~\cite{Fang2020} & Difference and two-term sum functions on
			$\mathbb Z$; both target functions are eventually at least $2$ &
			$A(-x,x)/\sqrt{x}$ has limsup at least $1/\sqrt2$ \\
			Fang~\cite{Fang2021} & Increasing $f:\N\to\N$ with $f\geq2$ &
			$A(x)=B(x/3)+O(\omega(x))$ \\
			Chen--Fang~\cite{ChenFang2027} & $f\equiv1$ on $\N$ &
			$A(x)=B(x/2)+O(\omega(x))$ \\
			Theorem~\ref{thm:main} & $\liminf f\geq2$, without monotonicity &
			$A(x)=B(x/2)+O(\omega(x))$ \\
			Theorem~\ref{thm:half} & $S_f$ contains long intervals centered at
			infinitely many elements of $2B$ & $A(x)=B(x/2)+O(\omega(x))$ \\
			Theorem~\ref{thm:third} & $S_f$ contains arbitrarily long intervals &
			$A(x)=B(x/3)+O(\omega(x))$ \\
			\hline
		\end{tabular}
		\caption{Prescribed difference functions and the resulting density scales.}
		\label{tab:comparison}
	\end{table}
	
	The interval hypotheses in Theorems~\ref{thm:half} and
	\ref{thm:third} are sufficient conditions.  We do not claim that they
	are necessary, nor that the dilations $2$ and $3$ are optimal under
	their respective hypotheses.  In particular, it remains open whether
	the assumption that $S_f$ contains arbitrarily long intervals alone
	always permits the scale $B(x/2)$.
	
	All three results use the same core-and-repair principle.  At each
	stage four facts are maintained: no representation number exceeds its
	target, the endpoints of every completed representation are protected,
	only a controlled number of core elements are removed, and all changes
	are made beyond a prescribed threshold.  A remote pair supplies one
	deficient difference.  The safe intervals in $S_f$ control its
	cross-differences, after which deletions from the Sidon core remove all
	excess.  For the even core $2B$, the required safe interval is aligned
	with a core element; the ternary core $3B$ permits the inserted pair to
	be placed in an arbitrary sufficiently long safe interval.  This is
	the source of the two density scales.  Sections~2 and~4 give the
	corresponding repair lemmas.
	Theorem~\ref{thm:half} is proved in Section~3, and the infinite case of
	Theorem~\ref{thm:third} is proved in Section~4.  Finite cores and
	Theorem~\ref{thm:main} are treated in Section~5.
	
	\section{A repair lemma for an even Sidon core}
	
	The next lemma is the main step of the construction.  Its role is to
	increase one deficient value of the representation function without
	allowing any value to exceed its prescribed upper bound.  At the same
	time, all previously protected representations remain untouched.
	
	\begin{lemma}[finite repair]\label{lem:repair}
		Let $f:\N\to\N$, and let $E$ be an infinite Sidon set consisting of
		even positive integers.  Suppose that, for every $T\in\N$, there are
		infinitely many $e\in E$ such that
		\[
		[e-T,e+T]\cap\N\subseteq S_f.
		\tag{2.1}\label{eq:E-compatibility}
		\]
		Fix $L\in\N$.  Let $O$ be a finite set of odd positive integers and let
		$C$ be a finite subset of $E$.  Put $D=O\cup(E\setminus C)$ and suppose
		that
		\[
		d_D(n)\leq f(n)\qquad(n\in\N).
		\tag{2.2}\label{eq:repair-upper}
		\]
		Let $u\in\N$ satisfy
		\[
		d_D(u)<f(u).
		\tag{2.3}\label{eq:target-deficient}
		\]
		Let $P$ be a finite subset of $D$.
		Then there exist a set $X$ of odd positive integers and a set
		$R\subseteq E\setminus C$ such that, on putting
		$D_1=(O\cup X)\cup\bigl(E\setminus(C\cup R)\bigr)$, the following
		properties hold:
		\begin{enumerate}
			\item[(i)] $\card{X}\leq2$, $\card{R}\leq10(\card{O}+2)$, and
			$X\cup R\subseteq(L,\infty)$;
			\item[(ii)] $P\subseteq D_1$ and
			$d_{D_1}(n)\leq f(n)$ for every $n\in\N$;
			\item[(iii)] $d_{D_1}(u)=d_D(u)+1$.
		\end{enumerate}
	\end{lemma}
	
	\begin{proof}
		Enlarge $P$, if necessary, by adjoining $D\cap[1,L]$ and the endpoints
		of every representation of $u$ in $D$, and continue to denote the
		enlarged set by $P$.  The enlarged set is finite because
		$d_D(u)\leq f(u)<\infty$.
		
		Choose an integer $M$ such that
		\[
		M>\max\bigl(P\cup O\cup C\cup\{L,u,1\}\bigr).
		\tag{2.4}\label{eq:M-choice}
		\]
		By \eqref{eq:E-compatibility}, we may choose $e\in E$ such that
		\[
		e>4M,\qquad [e-2M,e+2M]\cap\N\subseteq S_f,
		\qquad e+2u\notin E.
		\tag{2.5}\label{eq:anchor-choice}
		\]
		Indeed, the first two conditions hold for infinitely many $e\in E$,
		whereas the Sidon property permits at most one $e\in E$ with
		$e+2u\in E$.
		We divide the proof into three steps.
		
		\medskip
		\noindent
		\emph{Step 1: inserting one new representation of $u$.}
		
		Suppose first that $u$ is odd.  Since $e>M>\max C$, we have
		$e\in E\setminus C$.  Set
		\[
		X=\{e+u\},\qquad H=\{e,e+u\}.
		\tag{2.6}\label{eq:odd-block}
		\]
		The new element $e+u$ is odd.  Any new representation of the odd
		integer $u$ must pair $e+u$ with an even element.  The equation
		\[
		\lvert(e+u)-y\rvert=u
		\]
		has only the two possible solutions $y=e$ and $y=e+2u$.  The second
		does not belong to $E$ by \eqref{eq:anchor-choice}, while the first gives
		the intended representation
		\[
		(e+u)-e=u.
		\]
		Thus exactly one new representation of $u$ is created.
		
		Suppose next that $u$ is even.  Set
		\[
		X=H=\{e+1,e+1+u\}.
		\tag{2.7}\label{eq:even-block}
		\]
		Both new elements are odd and
		\[
		(e+1+u)-(e+1)=u.
		\]
		A representation of the even integer $u$ cannot use one odd and one
		even endpoint.  Hence any additional new representation of $u$ would
		have to join one of $e+1,e+1+u$ to an old element of $O$.  The only
		possible old candidates are $e+1-u$ and $e+1+2u$, but both are larger
		than $M$ and
		therefore do not belong to $O$.  Again, exactly one new representation
		of $u$ is created.
		
		In both parity cases, $X$ contains at most two odd integers, every
		element of $H$ is larger than $4M$, every cross-difference between $H$
		and a number smaller than $M$ belongs to $S_f$, and the two elements of
		$H$ have difference $u$.  Put
		\[
		D^*=D\cup X.
		\]
		
		\medskip
		\noindent
		\emph{Step 2: bounding the total excess created by $X$.}
		
		For a positive integer $n$, call the representations of $n$ that use
		at least one element of $X$ \emph{new}.  Since \eqref{eq:repair-upper}
		holds, every unit of excess
		\[
		\bigl(d_{D^*}(n)-f(n)\bigr)_+
		\]
		must be caused by a new representation.
		
		Write
		\[
		s=\card{X}\leq 2,\qquad q=\card{O}.
		\]
		If $n$ is even, a new representation of $n$ must use two odd
		endpoints.  Therefore its endpoints consist either of one element of
		$X$ and one element of $O$, or of two distinct elements of $X$.
		Consequently, the total number of new representations of even
		differences is at most
		\[
		sq+\binom{s}{2}.
		\tag{2.8}\label{eq:even-new-count}
		\]
		For each even $n$, the old bound $d_D(n)\leq f(n)$ implies that the
		excess at $n$ is at most the number of new representations of $n$.
		Thus \eqref{eq:even-new-count} also bounds the total excess over all
		even differences.
		
		Now let $n$ be odd.  Every representation of $n$ has one odd endpoint
		and one even endpoint.  A new representation may thus be written as
		the pair $\{x,e_1\}$ with
		\[
		x\in X,\qquad e_1\in E\setminus C,\qquad
		\lvert x-e_1\rvert=n.
		\]
		If it has the same difference as another representation
		$\{y,e_2\}$, where
		\[
		y\in O\cup X,\qquad e_2\in E\setminus C,
		\]
		then
		\[
		\lvert x-e_1\rvert=\lvert y-e_2\rvert.
		\]
		Removing the absolute values gives one of the two equations
		\[
		e_1-e_2=x-y
		\tag{2.9}\label{eq:collision-difference}
		\]
		or
		\[
		e_1+e_2=x+y.
		\tag{2.10}\label{eq:collision-sum}
		\]
		Since $E$ is Sidon, for fixed $x$ and $y$ equation
		\eqref{eq:collision-difference} has at most one nontrivial ordered
		solution.  When $x=y$, its only solutions have $e_1=e_2$ and correspond
		to the same representation, so they are not collisions.  For the sum
		equation, suppose that $e_1+e_2=e'_1+e'_2$ and that the two unordered
		pairs are distinct.  After relabeling, $e_1<e'_1$ and $e_2>e'_2$, so
		$e'_1-e_1=e_2-e'_2$ gives two representations of the same positive
		difference, contrary to the Sidon property.  Hence a fixed sum has at
		most one unordered representation in $E$, and therefore at most two
		ordered representations.  Thus, for every fixed pair
		$(x,y)\in X\times(O\cup X)$, there are at most three nontrivial
		collisions.
		
		It remains to explain why counting collisions controls the excess.
		Fix an overfull odd difference $n$, and let $a$ and $b$ be respectively
		the numbers of old and new representations of $n$.  Then
		\[
		a\leq f(n),\qquad a+b>f(n)\geq1.
		\]
		The number of unordered pairs of distinct representations of $n$ in
		which at least one representation is new is
		\[
		ab+\binom{b}{2}.
		\]
		If $a\geq1$, then
		\[
		a+b-f(n)\leq b\leq ab+\binom{b}{2}.
		\]
		If $a=0$, then $b>f(n)\geq1$, so $b\geq2$ and
		\[
		b-f(n)\leq b-1\leq\binom{b}{2}.
		\]
		Hence the excess at $n$ is at most the number of collisions involving
		a new representation.  Summing over all odd $n$ and using
		\eqref{eq:collision-difference}--\eqref{eq:collision-sum}, we obtain
		the bound $3s(q+s)$ for the total odd excess.
		
		Combining this with \eqref{eq:even-new-count} gives
		\begin{equation}
			\begin{aligned}
				\sum_{n\geq1}\bigl(d_{D^*}(n)-f(n)\bigr)_+
				&\leq sq+\binom{s}{2}+3s(q+s)\\
				&\leq 10(q+2).
			\end{aligned}
			\tag{2.11}\label{eq:total-excess}
		\end{equation}
		Indeed, $s\leq2$ makes the first line at most
		$2q+1+6(q+2)=8q+13$, which is no larger than $10(q+2)$.
		In particular, only finitely many differences are overfull.
		
		\medskip
		\noindent
		\emph{Step 3: deleting even elements that cause excess.}
		
		The elements of $P\cup H$ will never be deleted.  We first prove that
		whenever a current set has an overfull difference, some representation
		of that difference contains a deletable even element; that is, an element of
		\[
		(E\setminus C)\setminus(P\cup H).
		\tag{2.12}\label{eq:deletable-elements}
		\]
		
		Suppose, to the contrary, that every representation of some overfull
		difference $n$ avoids the set in \eqref{eq:deletable-elements}.  Then
		all its endpoints lie in
		\[
		N=S\cup H,\qquad S=O\cup(P\cap E).
		\]
		Every element of $S$ is smaller than $M$, every element of $H$ is
		larger than $4M$, and the two elements of $H$ have difference
		$u<M$.  We consider three ranges for $n$.
		
		If $n\leq M$, a pair joining $S$ to $H$ cannot represent $n$, because
		the distance between these two sets is greater than $3M$.  The only
		new representation inside $N$ is the pair $H$ at the target difference
		$u$.  Therefore
		\[
		d_N(n)\leq d_D(n)+\mathbf 1_{n=u}\leq f(n),
		\]
		where the last inequality follows from
		\eqref{eq:repair-upper} and \eqref{eq:target-deficient}.
		
		If $M<n\leq3M$, no pair inside $N$ represents $n$: differences inside
		$S$ and inside $H$ are at most $M$, whereas differences between $S$
		and $H$ are greater than $3M$.
		
		Finally, if $n>3M$, every representation inside $N$ joins one element
		of $H$ to one element of $S$.  Once the endpoint in $H$ is fixed, the
		other endpoint is uniquely determined.  Moreover, the choice of $e$
		in \eqref{eq:anchor-choice} gives
		\[
		n\in[e-2M,e+2M]\cap\N\subseteq S_f.
		\]
		Since $H$ has two elements,
		\[
		d_N(n)\leq2\leq f(n).
		\]
		Thus $N$ has no overfull difference in any of the three ranges, a
		contradiction.  This proves the claim.
		
		Starting from $D^*$, choose an overfull difference and delete one even
		element supplied by the claim.  Deletion cannot increase any
		representation number.  Moreover, the positive quantity on the
		left-hand side of \eqref{eq:total-excess} decreases by at least one,
		because the chosen deletion destroys a representation of an overfull
		difference.  The process must therefore stop after at most $10(q+2)$
		deletions.  Let $R$ be the set of deleted elements and put
		\[
		D_1=(O\cup X)\cup\bigl(E\setminus(C\cup R)\bigr).
		\]
		Then $d_{D_1}(n)\leq f(n)$ for all $n$ and
		\[
		\card{R}\leq10(q+2)=10(\card{O}+2).
		\]
		
		No element of $P\cup H$ was deleted.  Hence all old representations of
		$u$ survive, and the new representation with endpoint set $H$ also
		survives.  Step 1 showed that this was the only new representation of
		$u$, so
		\[
		d_{D_1}(u)=d_D(u)+1.
		\]
		Finally, if a deleted element were at most $L$, it would belong to
		$D\cap[1,L]\subseteq P$, contrary to the rule that $P$ is protected.
		Thus every element of $R$ is larger than $L$.  The same is true of
		every element of $X$ by \eqref{eq:M-choice}.  All assertions of the
		lemma now follow.
	\end{proof}
	
	\section{Proof of Theorem~\ref{thm:half}}
	
	Put
	\[
	E=2B=\{2b:b\in B\}.
	\]
	Then $E$ is an infinite Sidon set consisting of even positive integers,
	and
	\[
	E(x)=B(x/2).
	\tag{3.1}\label{eq:E-count}
	\]
	Condition \eqref{eq:compatibility} says precisely that $E$ satisfies
	\eqref{eq:E-compatibility}.
	
	Because $\omega(x)\to\infty$, there is a strictly increasing sequence
	of positive integers
	\[
	L_1<L_2<L_3<\cdots
	\]
	such that, for every $h\geq1$,
	\[
	\omega(x)>10h(h+1)\qquad(x>L_h).
	\tag{3.2}\label{eq:omega-schedule}
	\]
	
	We construct increasing sequences of finite sets
	\[
	\varnothing=O_0\subseteq O_1\subseteq O_2\subseteq\cdots,
	\qquad
	\varnothing=C_0\subseteq C_1\subseteq C_2\subseteq\cdots\subseteq E,
	\]
	and
	\[
	\varnothing=P_0\subseteq P_1\subseteq P_2\subseteq\cdots.
	\]
	Here $O_h$ consists of the odd integers inserted by the end of stage
	$h$, $C_h$ consists of the even elements deleted from $E$, and $P_h$
	is the set of elements that may not be deleted at any later stage.  Put
	\[
	D_h=O_h\cup(E\setminus C_h).
	\tag{3.3}\label{eq:stage-set}
	\]
	At every completed stage $h\geq1$, with $u_h$ denoting the difference
	treated at that stage, we require the following four properties:
	\begin{enumerate}
		\item
		\[
		d_{D_h}(n)\leq f(n)\qquad(n\in\N);
		\tag{3.4}\label{eq:stage-upper}
		\]
		\item $P_h\subseteq D_h$, $D_h\cap[1,L_h]\subseteq P_h$, and $P_h$
		contains both endpoints of every representation of every
		$n\leq u_h$ in $D_h$; moreover,
		\[
		d_{P_h}(n)=d_{D_h}(n)=f(n)
		\qquad(1\leq n<u_h);
		\]
		\item
		\[
		\begin{aligned}
			\card{O_h}&\leq2h,\\
			\card{C_h}&\leq10h(h+1);
		\end{aligned}
		\tag{3.5}\label{eq:stage-size}
		\]
		\item every element of $O_h\setminus O_{h-1}$ and
		$C_h\setminus C_{h-1}$ is larger than $L_h$.
	\end{enumerate}
	
	The second requirement is the permanent lower-bound invariant.  Since
	$P_h\subseteq P_k\subseteq D_k$ whenever $k\geq h$ and both stages
	are constructed, every
	representation counted by $d_{P_h}(n)$ survives all later stages.  In
	particular, once $d_{P_h}(n)=f(n)$ has been achieved, those $f(n)$
	representations will also belong to the final set.
	
	For $h=0$, set $D_0=E$.  Since $E$ is Sidon and $f(n)\geq1$,
	\eqref{eq:stage-upper} holds at stage $0$, while the size bounds are
	immediate.
	
	Assume that $D_{h-1}$ and $P_{h-1}$ have been constructed.  If
	$d_{D_{h-1}}(n)=f(n)$ for every $n$, no further modification is needed;
	we stop the construction at stage $H=h-1$.  If the construction has
	not stopped, let $u_h$ be the least positive integer for which
	\[
	d_{D_{h-1}}(u_h)<f(u_h).
	\tag{3.6}\label{eq:least-deficiency}
	\]
	Before applying the repair lemma, enlarge the protected set to
	\begin{align*}
		\widehat P_{h-1}={}&P_{h-1}\cup(D_{h-1}\cap[1,L_h])\\
		&{}\cup\bigcup_{1\leq n\leq u_h}
		\{a,a'\in D_{h-1}:a-a'=n\}.
		\tag{3.7}\label{eq:enlarged-protected-set}
	\end{align*}
	This set is finite.  Indeed, $D_{h-1}\cap[1,L_h]$ is finite, and
	\eqref{eq:stage-upper} implies that the last union contains at most
	\[
	2\sum_{n=1}^{u_h}f(n)
	\]
	elements.
	
	Apply Lemma~\ref{lem:repair} with
	\[
	O=O_{h-1},\quad C=C_{h-1},\quad
	P=\widehat P_{h-1},\quad u=u_h,
	\quad L=L_h.
	\]
	Let $X_h$ and $R_h$ be the sets supplied by the lemma.  Define
	\[
	\begin{aligned}
		O_h&=O_{h-1}\cup X_h,\\
		C_h&=C_{h-1}\cup R_h.
	\end{aligned}
	\tag{3.8}\label{eq:stage-update}
	\]
	and define $D_h$ by \eqref{eq:stage-set}.  Finally, let
	$P_h$ consist of $\widehat P_{h-1}$ together with both endpoints of
	every representation of every $n\leq u_h$ in $D_h$.
	The repair lemma gives \eqref{eq:stage-upper} and
	\[
	d_{D_h}(u_h)=d_{D_{h-1}}(u_h)+1.
	\tag{3.9}\label{eq:target-increment}
	\]
	The upper bound \eqref{eq:stage-upper} shows that $P_h$ is finite.  The
	repair lemma also shows that every element of $\widehat P_{h-1}$
	remains in $D_h$.  Hence $P_{h-1}\subseteq P_h\subseteq D_h$.  Moreover, all
	changes at stage $h$ occur above $L_h$, so
	\[
	D_h\cap[1,L_h]=D_{h-1}\cap[1,L_h]\subseteq P_h.
	\]
	For $n<u_h$, the definition of $u_h$ gives
	$d_{D_{h-1}}(n)=f(n)$.  All endpoints of these representations belong
	to $\widehat P_{h-1}$ and hence remain in $D_h$.  Together with
	\eqref{eq:stage-upper}, this gives
	\[
	d_{D_h}(n)=f(n)\qquad(1\leq n<u_h).
	\]
	Since $P_h$ contains the endpoints of every such representation, we
	also have $d_{P_h}(n)=f(n)$ in this range.  Thus requirement 2 holds,
	while the location statement in the repair lemma verifies requirement
	4.
	
	We verify \eqref{eq:stage-size}.  By induction,
	\[
	\card{O_h}\leq\card{O_{h-1}}+2\leq2h.
	\]
	Furthermore, Lemma~\ref{lem:repair} and the induction hypothesis give
	\begin{align*}
		\card{C_h}
		&\leq\card{C_{h-1}}+10(\card{O_{h-1}}+2)\\
		&\leq10(h-1)h+10(2h-2+2)\\
		&=10h(h+1).
	\end{align*}
	Thus requirement 3 also holds, and the induction step is complete.
	
	Requirement 2 shows that no integer smaller than $u_h$ becomes
	deficient.  At $u_h$, the representation number increases by exactly
	one according to \eqref{eq:target-increment}.  It follows that the
	successive values $u_h$ are nondecreasing.
	
	If the same value $u$ occurs at several consecutive stages, its
	representation number increases by one at each of those stages.  Since
	$f(u)$ is finite and the upper bound \eqref{eq:stage-upper} is never
	violated, $u$ can occur only finitely many times.  Thus, if the
	construction does not terminate, then $u_h\to\infty$.
	
	If the construction terminates at stage $H$, put $O=O_H$ and $C=C_H$.
	Otherwise, put $O=\bigcup_{h\geq0}O_h$ and
	$C=\bigcup_{h\geq0}C_h$.  In both cases define
	\[
	A=O\cup(E\setminus C).
	\tag{3.10}\label{eq:limit-set}
	\]
	For every constructed stage $h$,
	\[
	P_h\subseteq A.
	\]
	Indeed, a protected element remains in every later stage that is
	actually constructed.  If it is odd, it belongs to $O$; if it is even,
	it is never placed in $C$.
	
	If the construction terminates, its terminal set already satisfies
	$d_A(n)=f(n)$ for every $n$.  Otherwise $u_h\to\infty$.  Given
	$n\in\N$, choose $h$ with $u_h>n$.  Requirement 2 gives
	\[
	d_{P_h}(n)=f(n).
	\]
	Since $P_h\subseteq A$, it follows that
	\[
	d_A(n)\geq f(n)\qquad(n\in\N).
	\tag{3.11}\label{eq:limit-lower}
	\]
	To prove the reverse inequality, suppose that $A$ contained
	$f(n)+1$ distinct representations of some $n$.  These representations
	have only finitely many endpoints.  Every odd endpoint belongs to some
	$O_h$, so all odd endpoints belong to one common $O_h$ when $h$ is
	sufficiently large.  Every even endpoint belongs to $E\setminus C$ and
	therefore to $E\setminus C_h$ for every $h$.  Consequently, all
	$f(n)+1$ representations would already occur together in $D_h$ for
	some $h$, contradicting \eqref{eq:stage-upper}.  Thus
	\[
	d_A(n)\leq f(n)\qquad(n\in\N).
	\tag{3.12}\label{eq:limit-upper}
	\]
	Thus, in the nonterminating case, for the index $h$ chosen above we
	have the explicit chain
	\[
	f(n)=d_{P_h}(n)\leq d_A(n)\leq f(n).
	\]
	In the terminating case the same conclusion holds by the stopping
	rule.  Therefore
	\[
	d_A(n)=f(n)\qquad(n\in\N).
	\tag{3.13}\label{eq:exact-limit}
	\]
	
	It remains to estimate the counting function of $A$.  Fix $x\geq1$.
	If there is no $h\geq1$ for which $L_h<x$, then all modifications were
	made above $x$, and hence
	\[
	O(x)=C(x)=0.
	\]
	Otherwise, let
	\[
	j=\max\{h\geq1:L_h<x\}.
	\]
	Every modification made at a stage $h>j$ lies above $L_h\geq x$ and
	does not affect the counting functions at $x$; if the construction
	terminates before such a stage, there are simply fewer modifications.
	Therefore
	\[
	O(x)\leq2j\leq10j(j+1)<\omega(x),
	\qquad
	C(x)\leq10j(j+1)<\omega(x),
	\tag{3.14}\label{eq:limit-density-errors}
	\]
	where the strict inequalities follow from
	\eqref{eq:omega-schedule}.  Since $O$ is odd, $E$ is even, and
	$C\subseteq E$, equations \eqref{eq:E-count} and
	\eqref{eq:limit-set} give
	\[
	A(x)=E(x)+O(x)-C(x)
	=B(x/2)+O(x)-C(x).
	\]
	Together with \eqref{eq:limit-density-errors}, this yields
	\[
	B(x/2)-\omega(x)\leq A(x)\leq B(x/2)+\omega(x)
	\]
	for every $x\geq1$.  This proves Theorem~\ref{thm:half}.
	
	\section{A repair lemma for a ternary Sidon core}
	
	We now assume that $S_f$ contains arbitrarily long intervals.  Using a
	core contained in $3\N$ allows both endpoints of a newly inserted pair
	to be chosen outside the residue class of the core.
	
	\begin{lemma}[ternary repair]\label{lem:ternary-repair}
		Let $f:\N\to\N$ and suppose that $S_f$ contains arbitrarily long
		intervals.  Fix $L\in\N$.  Let $E$ be an infinite Sidon set contained
		in $3\N$, let $Q$ be a finite subset of $\N\setminus3\N$, and let
		$C$ be a finite subset of $E$.  Put $D=Q\cup(E\setminus C)$ and
		suppose that $d_D(n)\leq f(n)$ for every $n\in\N$.  Let $u\in\N$
		satisfy $d_D(u)<f(u)$, and let $P$ be a finite subset of $D$.
		Then there exist a two-element set $X\subseteq\N\setminus3\N$ and a
		set $R\subseteq E\setminus C$ such that, on putting
		$D_1=(Q\cup X)\cup\bigl(E\setminus(C\cup R)\bigr)$, the following
		properties hold:
		\begin{enumerate}
			\item[(i)] $\card{R}\leq4(\card{Q}+2)^2$ and
			$X\cup R\subseteq(L,\infty)$;
			\item[(ii)] $P\subseteq D_1$ and
			$d_{D_1}(n)\leq f(n)$ for every $n\in\N$;
			\item[(iii)] $d_{D_1}(u)=d_D(u)+1$.
		\end{enumerate}
	\end{lemma}
	
	\begin{proof}
		Enlarge $P$ by adjoining $D\cap[1,L]$ and the endpoints of every
		representation of $u$ in $D$.  This preserves finiteness.  Choose
		\[
		M>\max\bigl(P\cup Q\cup C\cup\{L,u,1\}\bigr).
		\tag{4.1}\label{eq:ternary-M}
		\]
		We first select the new pair.  Choose $r\in\{1,2\}$ such that
		$r+u\not\equiv0\pmod3$.  Hence neither $r$ nor $r+u$ is congruent to
		$0$ modulo $3$.
		
		We use the elementary local bound
		\[
		\binom{\card{E\cap I}}2\leq \ell
		\tag{4.2}\label{eq:local-Sidon}
		\]
		for every interval $I$ of length $\ell$.  Indeed, all positive
		differences between elements of $E\cap I$ are distinct and at most
		$\ell$.  If $t=\card{E\cap I}$, then
		$t(t-1)/2\leq\ell$, and hence
		\[
		t\leq\frac{1+\sqrt{1+8\ell}}2\leq1+\sqrt{2\ell}.
		\tag{4.3}\label{eq:local-cardinality}
		\]
		Choose $\ell$ so large that
		\[
		\left\lfloor\frac{\ell}{3}\right\rfloor
		>2\bigl(1+\sqrt{2\ell}\bigr).
		\tag{4.4}\label{eq:length-choice}
		\]
		Since $S_f$ contains arbitrarily long intervals, there is an integer
		$a$ such that
		\[
		[a,a+\ell+6M]\subseteq S_f.
		\tag{4.5}\label{eq:safe-interval}
		\]
		Put $J=[a+4M,a+4M+\ell]$.  Since $J-u$ and $J+2u$ are intervals of
		length $\ell$, \eqref{eq:local-cardinality} gives separately
		\[
		\card{(E+u)\cap J}=\card{E\cap(J-u)}\leq1+\sqrt{2\ell}
		\]
		and
		\[
		\card{(E-2u)\cap J}=\card{E\cap(J+2u)}\leq1+\sqrt{2\ell}.
		\]
		On the other hand, $J$ contains at least $\lfloor\ell/3\rfloor$
		integers congruent to $r$ modulo $3$.  By \eqref{eq:length-choice},
		we may therefore
		choose $z\in J$ such that
		\[
		z\equiv r\pmod3,\qquad z-u\notin E,
		\qquad z+2u\notin E.
		\tag{4.6}\label{eq:ternary-z}
		\]
		Moreover,
		\[
		z>4M,\qquad [z-M,z+M]\subseteq S_f.
		\tag{4.7}\label{eq:ternary-safe}
		\]
		
		Set $X=H=\{z,z+u\}$.  Both elements of $X$ lie outside $3\N$.
		The pair $H$ represents $u$.  Any other new representation of $u$
		would use $z-u$ or $z+2u$.  These two numbers do not belong to $E$ by
		\eqref{eq:ternary-z}, and they are larger than $M$, so they do not
		belong to $Q$.  Hence
		\[
		d_{D\cup X}(u)=d_D(u)+1.
		\tag{4.8}\label{eq:ternary-increment}
		\]
		
		We next bound the excess created by $X$.  Write $q=\card{Q}$ and
		$s=\card{X}=2$.  There are at most
		\[
		sq+\binom{s}{2}
		\tag{4.9}\label{eq:exceptional-new}
		\]
		new representations whose two endpoints lie in $Q\cup X$.  Consider
		a new representation with endpoints $x\in X$ and $e_1\in E\setminus C$.
		Its difference is not divisible by $3$, so it cannot coincide with a
		difference represented by two elements of $E$.
		
		If it has the same difference as a representation with endpoints
		$y\in Q\cup X$ and $e_2\in E\setminus C$, then
		\[
		e_1-e_2=x-y\qquad\text{or}\qquad e_1+e_2=x+y.
		\]
		For each fixed $(x,y)$, the Sidon property gives at most three
		nontrivial collisions, as in the proof of Lemma~\ref{lem:repair}.
		Their total number is at most $3s(q+s)$.  Finally, for fixed $x$ and a
		fixed representation inside $Q$, the equality of its difference with
		$|x-e_1|$ determines at most two possible values of $e_1$.  These
		collisions contribute at most $2s\binom q2$.
		
		We make the charging argument explicit.  For each $n\geq1$, let
		$a_n=d_D(n)$, let $b_n$ be the number of new representations of $n$
		whose endpoints both lie in $Q\cup X$, and let $c_n$ be the number of
		new representations with one endpoint in $X$ and the other in
		$E\setminus C$.  Then
		$d_{D\cup X}(n)=a_n+b_n+c_n$ and $a_n\leq f(n)$.  Since adding the
		$b_n$ internal representations can increase the positive part by at
		most $b_n$,
		\[
		\bigl(a_n+b_n+c_n-f(n)\bigr)_+
		\leq b_n+\bigl(a_n+c_n-f(n)\bigr)_+.
		\]
		If $a_n\geq1$, then the second term is at most $c_n$ and hence at most
		$a_nc_n+\binom{c_n}{2}$.  If $a_n=0$ and the second term is positive,
		then $c_n>f(n)\geq1$, so it is at most
		$c_n-1\leq\binom{c_n}{2}$.  Thus every remaining unit of excess is
		charged to a pair of equal-difference representations involving a new
		core-cross representation.
		
		Summing over $n$, the total of the $b_n$ is at most
		\eqref{eq:exceptional-new}.  The collisions with a representation
		using an endpoint in $Q\cup X$ contribute at most $3s(q+s)$, and those
		with an old representation inside $Q$ contribute at most
		$2s\binom q2$.  Consequently,
		\begin{align*}
			\sum_{n\geq1}\bigl(d_{D\cup X}(n)-f(n)\bigr)_+
			&\leq sq+\binom{s}{2}+3s(q+s)+2s\binom q2\\
			&=2q^2+6q+13\\
			&\leq4(q+2)^2.
			\tag{4.10}\label{eq:ternary-excess}
		\end{align*}
		In particular, only finitely many differences are overfull.
		
		It remains to eliminate this excess.  The elements of $P\cup H$ will
		not be deleted.  Suppose that all representations of an overfull
		difference avoid the deletable set $(E\setminus C)\setminus P$.  Their
		endpoints then lie in $N=S\cup H$, where
		$S=Q\cup(P\cap E)$.  Every element of $S$ is smaller than $M$ and
		every element of $H$ is larger than $4M$.  Consequently,
		\[
		d_N(n)\leq
		\begin{cases}
			d_D(n)+\mathbf1_{n=u},&n\leq M,\\
			0,&M<n\leq3M,\\
			2,&n>3M.
		\end{cases}
		\tag{4.11}\label{eq:ternary-N}
		\]
		The first line is at most $f(n)$ by the deficiency at $u$.  In the
		last line, every represented difference joins $S$ to $H$ and hence
		belongs to $[z-M,z+M]\subseteq S_f$ by
		\eqref{eq:ternary-safe}; thus $2\leq f(n)$.  Therefore $N$ has no
		overfull difference, a contradiction.
		
		Starting from $D\cup X$, repeatedly delete a core element belonging to
		a representation of an overfull difference.  Each deletion decreases
		the quantity in \eqref{eq:ternary-excess} by at least one, so the
		process stops after at most $4(q+2)^2$ deletions.  Let $R$ be the set
		of deleted elements.  No element of $P\cup H$ is deleted; hence
		\eqref{eq:ternary-increment} survives and $P\subseteq D_1$.  If an
		element of $R$ were at most $L$, it would belong to
		$D\cap[1,L]\subseteq P$.  Thus $R\subseteq(L,\infty)$, and the same
		holds for $X$ by \eqref{eq:ternary-M} and \eqref{eq:ternary-z}.
	\end{proof}
	
	We now prove Theorem~\ref{thm:third} when $B$ is infinite.  Put
	$E=3B$, so that $E$ is an infinite Sidon set contained in $3\N$ and
	\[
	E(x)=B(x/3).
	\tag{4.12}\label{eq:ternary-core-count}
	\]
	Choose $L_1<L_2<\cdots$ such that
	\[
	\omega(x)>16h^3\qquad(x>L_h).
	\tag{4.13}\label{eq:ternary-schedule}
	\]
	
	We construct increasing finite sets $Q_h\subseteq\N\setminus3\N$,
	$C_h\subseteq E$, and $P_h$, and put
	\[
	D_h=Q_h\cup(E\setminus C_h).
	\tag{4.14}\label{eq:ternary-stage}
	\]
	The construction starts with $Q_0=C_0=P_0=\varnothing$.  At the end of
	stage $h$ we require
	\begin{enumerate}
		\item[(i)] $d_{D_h}(n)\leq f(n)$ for every $n\in\N$;
		\item[(ii)] $P_h\subseteq D_h$, $D_h\cap[1,L_h]\subseteq P_h$, and
		$P_h$ contains the endpoints of every representation of every
		$n\leq u_h$ in $D_h$; moreover,
		$d_{P_h}(n)=d_{D_h}(n)=f(n)$ for $1\leq n<u_h$;
		\item[(iii)] $\card{Q_h}\leq2h$ and $\card{C_h}\leq16h^3$;
		\item[(iv)] every change made at stage $h$ occurs above $L_h$.
	\end{enumerate}
	
	Given $D_{h-1}$, stop if every difference has its prescribed number of
	representations.  Otherwise, let $u_h$ be its least deficient
	difference.  Before
	applying Lemma~\ref{lem:ternary-repair}, enlarge $P_{h-1}$ by adjoining
	$D_{h-1}\cap[1,L_h]$ and the endpoints of every representation of
	every $n\leq u_h$.  The resulting set is finite because
	$d_{D_{h-1}}(n)\leq f(n)$.  Use this enlarged set as the protected set
	$P$ in Lemma~\ref{lem:ternary-repair}.  The lemma supplies sets $X_h$
	and $R_h$;
	define
	\[
	Q_h=Q_{h-1}\cup X_h,\qquad C_h=C_{h-1}\cup R_h,
	\]
	and let $P_h$ be the enlarged protected set together with the endpoints
	of all representations of every $n\leq u_h$ in $D_h$.  Properties
	(i), (ii), and (iv) follow
	from the lemma exactly as in Section~3.  Moreover,
	\[
	\card{Q_h}\leq2h,\qquad
	\card{C_h}\leq4\sum_{j=1}^h(2j)^2
	=16\sum_{j=1}^h j^2\leq16h^3,
	\tag{4.15}\label{eq:ternary-size}
	\]
	which proves (iii).
	
	The successive targets $u_h$ are nondecreasing, and each fixed value
	can occur only finitely many times.  Thus the construction either
	terminates at some stage $H$ with the prescribed representation
	function or, if it never terminates, satisfies $u_h\to\infty$.  In the
	terminal case put $Q=Q_H$ and $C=C_H$; otherwise put
	$Q=\bigcup_{h\geq0}Q_h$ and $C=\bigcup_{h\geq0}C_h$.  In both cases
	put
	\[
	A=Q\cup(E\setminus C).
	\tag{4.16}\label{eq:ternary-limit}
	\]
	Every protected set belonging to a constructed stage is contained in
	$A$.  If the process terminates, the stopping rule gives $d_A=f$.
	Otherwise, given $n$, choose $h$ with $u_h>n$.  Then
	\[
	d_A(n)\geq d_{P_h}(n)=f(n).
	\]
	If $A$ had $f(n)+1$ representations of $n$, their finitely many
	noncore endpoints would all lie in some $Q_h$, while every core
	endpoint would lie in $E\setminus C_h$.  Those representations would
	therefore occur together in $D_h$, contradicting (i).  Hence
	\[
	d_A(n)=f(n)\qquad(n\in\N).
	\tag{4.17}\label{eq:ternary-exact}
	\]
	
	For $x\geq1$, let $j$ be the largest index with $L_j<x$, if such an
	index exists.  Only stages with index at most $j$ affect the counting
	functions at $x$; in the terminal case there may be fewer such stages.
	Thus \eqref{eq:ternary-size} gives
	\[
	Q(x)\leq2j<\omega(x),\qquad
	C(x)\leq16j^3<\omega(x).
	\]
	If no such $j$ exists, both quantities are zero.  Since $Q$ is
	disjoint from $E$,
	\[
	A(x)=B(x/3)+Q(x)-C(x).
	\]
	This proves Theorem~\ref{thm:third} for infinite $B$.
	
	\section{Finite Sidon sets and proof of Theorem~\ref{thm:main}}
	
	For a finite core no deletion is needed.  The following elementary
	lemma is the positive-integer analogue of the remote insertion used by
	Chen and Fang~\cite{ChenFang2011}.
	
	\begin{lemma}[remote insertion]\label{lem:finite-insertion}
		Let $f:\N\to\N$ and suppose that $S_f$ contains arbitrarily long
		intervals.  Let $D\subseteq\N$ be finite and satisfy
		$d_D(n)\leq f(n)$ for every $n\in\N$.  If $d_D(u)<f(u)$ and
		$L\in\N$, then there is a finite set $D_1\supseteq D$ such that
		$D_1\setminus D$ consists of two elements larger than $L$,
		$d_{D_1}(u)=d_D(u)+1$, and $d_{D_1}(n)\leq f(n)$ for every $n\in\N$.
	\end{lemma}
	
	\begin{proof}
		Choose an integer $M>\max(D\cup\{L,u,1\})$.  There is an interval
		$[a,a+5M]\subseteq S_f$.  Put $z=a+2M$ and
		$D_1=D\cup\{z,z+u\}$.  For every $y\in D$,
		\[
		a+M<z-y\leq z+u-y<a+3M.
		\tag{5.1}\label{eq:finite-cross}
		\]
		Thus every cross-difference belongs to $S_f$ and is larger than every
		old difference and larger than $u$.  For a fixed difference there is
		at most one representation using $z$ and at most one using $z+u$.
		Consequently every cross-difference has at most two representations,
		whereas the pair $\{z,z+u\}$ increases the representation number of
		$u$ by exactly one.  All assertions follow.
	\end{proof}
	
	\begin{lemma}\label{lem:finite-core}
		Let $B$ be a finite Sidon set, let $m\in\{2,3\}$, and suppose that
		$S_f$ contains arbitrarily long intervals.  For every positive
		function $\omega(x)\to\infty$, there exists $A\subseteq\N$ such that
		\[
		d_A(n)=f(n)\quad(n\in\N),\qquad
		\bigl|A(x)-B(x/m)\bigr|\leq\omega(x)\quad(x\geq1).
		\]
	\end{lemma}
	
	\begin{proof}
		Put $E=mB$ and choose $L_1<L_2<\cdots$ such that
		$\omega(x)>2h$ whenever $x>L_h$.  Starting with $D_0=E$, let $u_h$ be
		the least deficient difference of $D_{h-1}$ and apply
		Lemma~\ref{lem:finite-insertion} with $L=L_h$.  This gives an increasing
		sequence of finite sets satisfying
		\[
		d_{D_h}(n)\leq f(n),\qquad
		d_{D_h}(u_h)=d_{D_{h-1}}(u_h)+1,\qquad
		\card{D_h\setminus E}\leq2h.
		\tag{5.2}\label{eq:finite-stage}
		\]
		No difference smaller than $u_h$ changes, so the targets are
		nondecreasing and eventually pass every positive integer.  For
		$A=\bigcup_{h\geq0}D_h$, the usual finite-endpoint argument gives
		$d_A(n)=f(n)$ for every $n$.
		
		If $j$ is the largest index with $L_j<x$, then stages after $j$ add no
		element below $x$, and
		\[
		0\leq A(x)-E(x)\leq2j<\omega(x).
		\]
		The same difference is zero when no such $j$ exists.  Since
		$E=mB\subseteq A$, we have $A(x)\geq E(x)=B(x/m)$ for every $x$;
		the preceding estimate gives
		$A(x)\leq B(x/m)+\omega(x)$.  This proves the stated two-sided bound.
	\end{proof}
	
	Lemma~\ref{lem:finite-core} with $m=3$ completes the proof of
	Theorem~\ref{thm:third} when $B$ is finite.
	
	\begin{proof}[Proof of Theorem~\ref{thm:main}]
		Choose $N_0$ such that $f(n)\geq2$ for $n>N_0$.  If $B$ is infinite,
		then for every $L$ all sufficiently large $b\in B$ satisfy
		\[
		[2b-L,2b+L]\cap\N\subseteq S_f,
		\]
		so Theorem~\ref{thm:half} applies.  If $B$ is finite, then $S_f$
		contains arbitrarily long intervals, and Lemma~\ref{lem:finite-core}
		with $m=2$ gives the required set.
	\end{proof}


\begin{thebibliography}{99}
		
		\bibitem{AjtaiKomlosSzemeredi}
		M. Ajtai, J. Koml\'os and E. Szemer\'edi,
		A dense infinite Sidon sequence,
		\emph{European J. Combin.} \textbf{2} (1981), 1--11.
		
		\bibitem{ChenFang2011}
		Y.-G. Chen and J.-H. Fang,
		On a problem of Cilleruelo and Nathanson,
		\emph{Combinatorica} \textbf{31} (2011), 691--696.
		
		\bibitem{ChenFang2027}
		Y.-G. Chen and J.-H. Fang,
		Dense perfect difference sets constructed from Sidon sets,
		\emph{J. Combin. Theory Ser. A} \textbf{225} (2027), Article 106239.
		
		\bibitem{CillerueloNathanson2008}
		J. Cilleruelo and M. B. Nathanson,
		Perfect difference sets constructed from Sidon sets,
		\emph{Combinatorica} \textbf{28} (2008), 401--414.
		
		\bibitem{CillerueloNathanson2013}
		J. Cilleruelo and M. B. Nathanson,
		Dense sets of integers with prescribed representation functions,
		\emph{European J. Combin.} \textbf{34} (2013), 1297--1306.
		
		\bibitem{Fang2020}
		J.-H. Fang,
		On generalized perfect difference sumsets,
		\emph{SIAM J. Discrete Math.} \textbf{34} (2020), 2448--2456.
		
		\bibitem{Fang2021}
		J.-H. Fang,
		On generalized perfect difference sets constructed from Sidon sets,
		\emph{Discrete Math.} \textbf{344} (2021), Article 112589.
		
		\bibitem{Kruckeberg}
		F. Kr\"uckeberg,
		$B_2$-Folgen und verwandte Zahlenfolgen,
		\emph{J. Reine Angew. Math.} \textbf{206} (1961), 53--60.
		
		\bibitem{Lev}
		V. F. Lev,
		Reconstructing integer sets from their representation functions,
		\emph{Electron. J. Combin.} \textbf{11} (2004), Research Paper 78.
		
		\bibitem{Nathanson2003}
		M. B. Nathanson,
		Unique representation bases for the integers,
		\emph{Acta Arith.} \textbf{108} (2003), 1--8.
		
		\bibitem{Nathanson2005}
		M. B. Nathanson,
		Every function is the representation function of an additive basis for
		the integers,
		\emph{Port. Math. (N.S.)} \textbf{62} (2005), 55--72.
		
		\bibitem{Pollington}
		A. D. Pollington,
		On the density of $B_2$-bases,
		\emph{Discrete Math.} \textbf{58} (1986), 209--211.
		
		\bibitem{PollingtonVandenEynden}
		A. D. Pollington and C. Vanden Eynden,
		The integers as differences of a sequence,
		\emph{Canad. Math. Bull.} \textbf{24} (1981), 497--499.
		
		\bibitem{Ruzsa1993}
		I. Z. Ruzsa,
		Solving a linear equation in a set of integers I,
		\emph{Acta Arith.} \textbf{65} (1993), 259--282.
		
		\bibitem{Ruzsa1998}
		I. Z. Ruzsa,
		An infinite Sidon sequence,
		\emph{J. Number Theory} \textbf{68} (1998), 63--71.
		
		\bibitem{Stohr}
		A. St\"ohr,
		Gel\"oste und ungel\"oste Fragen \"uber Basen der nat\"urlichen
		Zahlenreihe I, II,
		\emph{J. Reine Angew. Math.} \textbf{194} (1955), 40--65, 111--140.
		
	\end{thebibliography}
\end{document}